\documentclass{amsproc}

\usepackage{amsmath, tikz-cd}
\usepackage{multicol}
\usepackage{subfigure}
\usepackage{amssymb}
\usepackage{amsfonts}
\usepackage{enumitem}
\usepackage{amsthm}
\usepackage{xcolor}
\usepackage{mathtools}

\theoremstyle{definition}

\newtheorem*{Corollary*}{Corollary}
\newtheorem{Theorem}[equation]{Theorem}
\newtheorem*{Theorem*}{Theorem}
\newtheorem{Lemma}[equation]{Lemma}
\theoremstyle{definition}
\newtheorem{Definition}[equation]{Definition}

\theoremstyle{remark}

\numberwithin{equation}{section}

\begin{document}

\title{Ces\`aro-Type Operators Acting on the Drury-Arveson Space}

\author{Michael R. Pilla}
\address{Florida Polytechnic University, Lakeland, FL 33805}
\curraddr{Florida Polytechnic University, Lakeland, FL 33805}
\email{mpilla@floridapoly.edu}

\date{April 30, 2026.}


\begin{abstract}
The celebrated Ces\`aro  operator is a well-known operator with interesting connections to a variety of objects in operator theory. Generalizations have been made for Cesàro-type operators acting on weighted Hardy spaces but constructing analogs of the Cesàro operator for function spaces of several complex variables such as the Drury-Arveson space has yet to be achieved. In this article, we posit a definition we belief is the correct generalization to several variables and establish a few of its basic properties.
\end{abstract}

\maketitle

\section{Preliminaries}

\subsection{A Gap in the Literature}

The Ces\`aro operator is a well-known operator with deep connections to numerous objects of interest in operator theory. Generalizations of this classic operator have enjoyed a flurry of activity over the years (\cite{Rhaly}, \cite{Dellepiane}, etc.). These generalizations, however, primarily focus on weighted Hardy spaces. These are Hilbert spaces of analytic functions on the unit disk whose monomials $1, z, z^2,...$ form a complete orthogonal set of non-zero vectors. However, appropriate generalizations of the Ces\`aro operator to the canonical Hilbert function spaces in several complex variables has yet to be addressed. It is the present aim to fill this gap. The focus will be on the Drury-Arveson space as this space has important universal properties akin to those of the Hardy space in one variable. However, the interested reader may expand the results to other Hilbert function spaces of several complex variables. We seek an approach that captures the qualities possessed by the Ces\`aro operator acting on the Hardy space such that it's status as an appropriate generalization is justified.

\subsection{The Drury-Arveson Space and Friends}

We begin with a review of the celebrated Drury-Arveson space on the unit ball $H^2_d(\mathbb{B}_d)$. We assume knowledge of the Hardy space on the unit disk, denoted by $H^2(\mathbb{D})$ or just $H^2$ when clear from context. We first give a quick review of notation in several variables. Recall the following.

For $z=(z_1,...,z_d) \in \mathbb{C}^d$, we let 

\begin{equation}z^{\alpha}=\prod_{i=1}^d z_i^{\alpha_i}\end{equation}

\noindent for every multi-index $\alpha=(\alpha_1,..., \alpha_d) \in \mathbb{N}^d$. Likewise we write

\begin{equation}
\alpha !=\prod_{i=1}^d \alpha_i !
 \quad\text{and}\quad
|\alpha|=\sum_{i=1}^d \alpha_i.
\end{equation}

In one variable, the monomials $1, z, z^2,...$ form a natural ordering. In several variables, the monomials do not come equipped with such a natural ordering. A standard ordering is established using the multi-indices of the monomials $z^{\alpha}$ by simply endowing them with lexographical order. To better observe the pattern, we explicitly write out our ordering for $\mathbb{C}^2$, we obtain 

\begin{equation}
\{1, z_1, z_2, z_1^2, z_1z_2, z_2^2, z_1^3, z_1^2z_2, z_1z_2^2, z_2^3,...\}.
\end{equation}

This will be our assumed ordering.

We recall that a Hilbert function space, also known as a reproducing kernel Hilbert space (RKHS) is a Hilbert space of complex-valued functions with pointwise vector operations along with the property that for each $z$ in our set, the linear functional given by evaluation at $z$, $f \rightarrow f(z)$ is continuous. By the Riesz representation theorem, there exists a unique function $k_z$ called the kernel function in the Hilbert space that induces the linear functional $f(z)=\langle f, k_z \rangle$. In such a case, we call the function $k_z$ the reproducing kernel. 

Let $\mathbb{B}_d=\{z \in \mathbb{C}^d \mid |z|<1 \}$ be the unit ball. We define the Drury-Arveson space $H^2_d(\mathbb{B}_d)$ to be the Hilbert function space with kernel

\begin{equation}k(z,w)=k_w(z)=\frac{1}{1-\langle z, w \rangle}\end{equation}

\noindent where $\langle \cdot, \cdot \rangle$ is the standard inner product. For two functions $f,g \in H^2_d$ if we have Taylor expansions given by

\begin{equation}
f(z)=\sum_{\alpha} c_{\alpha}z^{\alpha}
\quad \text{and} \quad
g(z)=\sum_{\alpha} d_{\alpha}z^{\alpha}
\end{equation}

\noindent then we define their inner product to be

\begin{equation}\langle f, g \rangle_{H^2_d}=\langle f, g \rangle=\sum_{\alpha} \frac{\alpha!}{|\alpha|!}c_{\alpha} \overline{d_{\alpha}}.\end{equation}

An orthonormal basis is then given by $\{e_{\alpha}\}$ where

\begin{equation}e_{\alpha}=\sqrt{\frac{|\alpha|!}{\alpha!}}z^{\alpha}.\end{equation}

Notably, the set of monomials $\{z^{\alpha}\}$ form an orthogonal basis for $H^2_d(\mathbb{B}_d)$.

An analytic function $f(z)=\sum_{\alpha}c_{\alpha}z^{\alpha}$ is in $H^2_d$  if

\begin{equation}||f||^2_{\mathbb{H}^2_d}=||f||^2=\sum_{\alpha} |c_{\alpha}|^2\frac{\alpha !}{|\alpha|!}< \infty.\end{equation}

Note that for $w \in \mathbb{B}_d$, we have

\begin{equation}k(z,w)=k_w(z)=\frac{1}{1-\langle z, w \rangle}=\sum_{n=0}^{\infty} \langle z,w \rangle^n=\sum_{n=0}^{\infty} \sum_{|\alpha|=n}\frac{|\alpha|!}{\alpha !}\overline{w}^{\alpha}z^{\alpha}\end{equation}

\noindent which is in $H^2_d$ and

\begin{equation}f(w)=\sum_{\alpha}c_{\alpha}w^{\alpha}=\sum_{\alpha}c_{\alpha}\frac{|\alpha|!}{\alpha !}w^{\alpha}\langle z^{\alpha}, z^{\alpha} \rangle =\langle f, k_w \rangle.\end{equation}

Recall, we say a function is analytic in the unit ball if it can be written as an absolutely convergent power series $f(z)=\sum_{\alpha}a_{\alpha}z^{\alpha}$. It is often insightful to write this as a decomposition of homogeneous functions $f_n(z)=\sum_{|\alpha|=n} a_{\alpha}z^{\alpha}$ which then gives us $f(z)=\sum_{n=0}^{\infty} f_n(z)$. It immediately follows that $\langle f_n, g_m \rangle=0$ for $m \neq n$.

This space is the unique Hilbert function space satisfying the following three condtions, giving it merit as the appropriate analogue of $H^2$ in several variables:

\begin{itemize}
\item The polynomials are dense.
\item Unitary invariance.
\item Recovery of the Hardy space $H^2$ when restricting to functions of a single variable.
\end{itemize}

The Drury-Arveson space is one space on a scale of Hilbert function spaces known as the standard weighted Bergman spaces (or Hardy-Sobolev spaces) given by the reproducing kernels

\begin{equation}
k^{\alpha}(z,w)=\frac{1}{(1-\langle z, w \rangle)^{\alpha}}.
\end{equation}

For $d$-variables, we have $d=1$ the Szego kernel corresponding to the Drury-Arveson space. We have $\alpha=d$ corresponding to the generalized Hardy space and $\alpha=d+1$ corresponding to the generalized Bergman space. Let $\mathcal{O}(\mathbb{B}_d)$ denote the set of function of $d$ complex variables analytic on the unit ball $\mathbb{B}_d$. We may write the definitions in terms of power series as follows:

\begin{equation}
\mathcal{H}_a=\{f(z)=\sum_{\alpha}a_{\alpha}z^{\alpha} \in \mathcal{O}(\mathbb{B}_d) \mid \sum_{\alpha}|a_{\alpha}|^2\frac{\alpha!}{|\alpha|!}(|\alpha|+1)^{1-\alpha}<\infty\}.
\end{equation}

Our primary focus will be the case $\alpha=1$. For further reading on the Drury-Arveson space, the reader is referred to survey articles by M. Hartz \cite{Hartz} and O. Shalit \cite{Shalit} or the original articles by Drury and Arveson \cite{Drury}, \cite{Arveson}. For further reading on Hardy-Sobolev spaces, see \cite{Zhao}. 

\subsection{The Classical Ces\`aro Operator}

We next briefly recall the classical definition of the Ces\`aro operator acting on $H^2$ in search of motivation for our definition of a multivariable analogue. In such a pursuit, we first recall how the Ces\`aro Operator $C$ acts on the Hardy space. Recall that the monomials given by $\{z^n\}_{n=0}^{\infty}$ form an orthonormal basis for $H^2$. Thus, to understand $C$ in the context of the Hardy space, it is sufficient to determine the action of $C$ on this basis. We have the following:

\begin{align}
C1&=1+\frac{z}{2}+\frac{z^2}{3}+\cdots\\
Cz&=\frac{z}{2}+\frac{z^2}{3}+\frac{z^3}{4}\cdots\\
Cz^2&=\frac{z^2}{3}+\frac{z^3}{4}+\frac{z^4}{5}+\cdots
\end{align}

and in general we have $Cz^n=\sum_{k=n}^{\infty}\frac{1}{k+1}z^k$.

Basic operator theoretic properties of the Ces\`aro Operator have been well studied. A few highlights include the following:

\begin{itemize}
\item The inverse of $C$ is given by $(C^{-1}f)(z)=(1-z)\frac{d}{dz}(zf(z))$.
\item $C$ is bounded on $H^2(\mathbb{D})$ with $||C||=2$ \cite{Halmos}.
\item The spectrum is given by $\sigma(C)=\{z \mid |z-1| \leq 1 \}$ \cite {Halmos}.
\end{itemize}

The Ces\`aro Operator contines to enjoy a flurry of activity (\cite{Agler}, \cite{Gallardo}, \cite{Kizgut}, \cite{Lacruz},\cite{Persson}, \cite{Sin}) due to its deep connections to other well-known operators as well as being an interesting concrete example to study important operator theoretic features such as commutants and invariance subspaces.

When considering the Ces\`aro Operator acting on Hilbert function spaces in one variable such as the Hardy space, it is often insightful to utilize it's (equivalent) integral form given by

\begin{equation}
(Cf)(z)=\int_0^1 f(tz)\frac{1}{1-tz}dt.
\end{equation}

Such an expression affords one an alternative path to investigate operator theoretic properties such as its spectrum, boundedness, and so forth. A wealth of results are known for the Ces\`aro Operator. The reader is encouraged to see \cite{Ross} or \cite{Ross2} for excellent expositions on the subject.

\section{Introducing the Ces\`aro Tuple }

Perhaps the first obstruction in generalizing the notion of a Ces\`aro operator to function spaces of several variables that arises is the fact that there is not a natural map from $n$-degree homogeneous polynomials to $(n+1)$-degree homogeneous polynomials. One way around this that has found immense success is to introduce a {\it tuple} of operators. For example, the celebrated shift operator $M_z$ acting on the Hardy space generalizes to function spaces in $d$-variables via a tuple of shift operators $[M_{z_1}, M_{z_2},...,M_{z_d}]$ where $M_{z_j}$ is the shift operator in the $j$th coordinate with $1 \leq j \leq d$. With this in mind, we posit the following definition:

\begin{Definition}
Let $f=\sum_{n=0}^{\infty}f_n$ be the homogeneous expansion of $f \in H^2_d$. For $\{j\}_{1}^d$, we define a $d$-tuple of operators given by $[C_1,...,C_d]$ where, for $1 \leq j \leq d$,

\begin{equation}
C_jf=f_0+\frac{1}{2}(z_j f_0+f_1)+ \frac{1}{3} (z_j^2 f_0 + z_j f_1 + f_2) + \ldots 
\end{equation}

\end{Definition}

More formally, since the monomials form an orthogonal basis for $H^2_d$ and the each $C_j$ is a linear operator, we may define this operator via its action on the monomials $z^{\alpha}$. 

\begin{Definition}
We define our $d$-tuple of operators given by $[C_1,...,C_d]$ for $1 \leq j \leq d$ via their action on monomials by

\begin{equation}
C_jz^{\alpha}=\sum_{k=0}^{\infty}\frac{1}{|\alpha|+k+1}z^{\alpha+ke_j}.\label{def}
\end{equation}

\end{Definition}

As we shall demonstrate, $C_j$ is bounded on $H^2_d$ and thus these definitions are equivalent. Additionally, this definition admits many analogous properties as its one variable counterpart. For brevity, we shall refer to this tuple as the Ces\`aro tuple.

\section{Properties of our Ces\`aro Tuple}

\subsection{Boundedness of Each Operator in the Ces\`aro Tuple on $H^2_d(\mathbb{B}_d)$}

When discovering a new operator, a reasonable first question to ask is whether it is bounded. In our case, we affirm this is so. We began with a quick lemma.

\begin{Lemma}
Let $w_{\beta}=\frac{\beta!}{|\beta|!}$ and $1 \leq j \leq d$. Then $\frac{w_{\beta}}{w_{\beta-ke_j}} \leq 1$ for any nonnegative integer$k$ and  $\sum_{k=0}^{\beta_j} \frac{w_{\beta}}{w_{\beta-ke_j}}  \leq |\beta|+1$.
\end{Lemma}

\begin{proof}
By definition we have

\begin{align}
\frac{w_{\beta}}{w_{\beta-ke_j}}&=\frac{\beta}{|\beta|}\frac{(|\beta|-k)!}{(\beta-ke_j)!}=\frac{\beta_j!}{(\beta_j-k)!}\frac{(|\beta|-k)!}{|\beta|!}\\
&=\frac{\beta_j(\beta_j-1)\cdots(\beta_j-k+1)}{|\beta|(\beta|-1)\cdots (|\beta|-k+1)} \leq 1
\end{align}

since $\beta_j \leq |\beta|$.

This implies that $\sum_{k=0}^{\beta_j} \frac{w_{\beta}}{w_{\beta-ke_j}} \leq \beta_j+1 \leq |\beta|+1$. 

\end{proof}

\begin{Theorem}
For $1 \leq j \leq d$,  $C_j$ is bounded on $H^2_d(\mathbb{B}_d)$.
\end{Theorem}

\begin{proof}

We start with polynomials, which are dense in $H^2_d$, and use the Bounded Linear Transformation theorem to extend our result to the entire space $H^2_d$. 

First we define $C_j$ as in \eqref{def} on the space of polynomials. For a polynomial $f$ with $f(z)=\sum_{\alpha}a_{\alpha} z^{\alpha}$ where $|\alpha| \leq N$ for some nonnegative integer $N$, we have $||f||^2=\sum_{\alpha} |a_{\alpha}|^2 \frac{\alpha!}{|\alpha|!}$. Then

\begin{equation}
C_jf=\sum_{\alpha}a_{\alpha}\sum_{k=0}^{\infty}\frac{1}{|\alpha|+k+1}z^{\alpha+ke_j}.
\end{equation}

Next we note that $\beta=\alpha+ke_j$ implies $\alpha=\beta-ke_j$. For fixed $\beta$, we have that $0 \leq k \leq \beta_j$. Note that we may rearrange the terms in the second summation by the properties of $H^2_d$. Specifically, if we view our second summation as a function of only $z_j$, $C_j$ is shifting along a slice in the $z_j^{\text{th}}$-direction which converges in the $H^2$ norm. Hence, rearranging terms, we may write 

\begin{equation}
C_jf=\sum_{\beta}\frac{1}{|\beta|+1}\left(\sum_{k=0}^{\beta_j}a_{\beta-ke_j}\right) z^{\beta}
\end{equation}

which implies that

\begin{equation}
||C_jf||^2=\sum_{\beta}\left|\sum_{k=0}^{\beta_j}\frac{a_{\beta-ke_j}}{|\beta|+1} \right|^2\frac{\beta!}{|\beta|!}.
\end{equation}

Letting $b_{\beta}=\sum_{k=0}^{\beta_j}\frac{a_{\beta-ke_j}}{|\beta|+1}$, we want to show that 

\begin{equation}
\sum_{\beta}|b_{\beta}|^2\frac{\beta!}{|\beta|!} \leq M \sum_{\alpha}|a_{\alpha}|^2\frac{\alpha!}{|\alpha|!}
\end{equation}

for some constant $M>0$. We aim to single out our weighted shift in the $j^{\text{th}}$ direction. Let $\beta=(\beta_j, \beta')$ where $\beta'$ denotes the remaining coordinates. We then have $b_{\beta}=b_{(\beta_j, \beta')}=\frac{1}{\beta_j+|\beta'|+1}\sum_{k=0}^{\beta_j}a_{(\beta_j-k,\beta')}$.

Fixing $\beta'$, we observe that

\begin{equation}
|b_{(\beta_j, \beta')}|\leq \frac{1}{\beta_j+1}\sum_{k=0}^{\beta_j}|a_{(k, \beta')}|
\end{equation}

which implies for fixed $\beta'$ that 

\begin{equation}
\sum_{\beta_j=0}^{\infty} |b_{(\beta_j, \beta')}|^2 \frac{\beta_j!\beta'!}{(\beta_j+|\beta'|)!}\leq \sum_{\beta_j=0}^{\infty} \left(\frac{1}{\beta_j+1}\sum_{k=0}^{\beta_j}|a_{(k, \beta')}|\right)^2 \frac{\beta_j!\beta'!}{(\beta_j+|\beta'|)!}.
\end{equation}

Next we recall the weighted Hardy inequality which states for a nonnegative sequence $\{x_n\}$ and a nonincreasing sequence of weights $\{w_n\}$ that

\begin{equation}
\sum_0^{\infty}\left(\frac{1}{n+1}\sum_{k=0}^n x_k \right)^2w_n \leq 4 \sum_{n=0}^{\infty} x_n^2 w_n.
\end{equation}

Note that by our lemma, our weights are nonincreasing and we may apply Hardy's weighted inequality. Applying this to the above expression and summing over all possible $\beta'$'s, we obtain

\begin{equation}
||C_jf||^2=\sum_{\beta'}\sum_{\beta_j=0}^{\infty}|b_{(\beta_j, \beta')}|^2\frac{\beta_j!\beta'!}{(\beta_j+|\beta'|)!} \leq 4 \sum_{\beta'}\sum_{\beta_j=0}^{\infty}|a_{(\beta_j, \beta')}|^2\frac{\beta_j!\beta'!}{(\beta_j+|\beta'|)!}.
\end{equation}

which implies

\begin{equation}
||C_jf||^2 \leq 4||f||^2.
\end{equation}

as desired. Finally, since our constant $M=4$ is not dependent on the degree of the polynomial and $C_j$ is bounded on the set of polynomials which are dense in $H^2_d$, we may apply the Bounded Linear Transformation theorem to conclude our result.

\end{proof}

\subsection{An Integral Approach for the Ces\`aro Tuple Acting on Spaces of Analytic Functions}

Let $\mathcal{O}(\mathbb{B}_d)$ denote the vector space of analytic functions on $\mathbb{B}_d$. We next produce an integral form for $C_j$ that parallels the form in one variable, further justifying our definition as the appropriate generalization of the Ces\`aro operator. 

\begin{Theorem}
 Given $f \in \mathcal{O}(\mathbb{B}_d)$, the integral form for $C_j$ is given by 

\begin{equation}
C_j f(z)=\int_0^1 f(tz)\frac{1}{1-tz_j}dt
\end{equation}

\end{Theorem}

\begin{proof}
Given $|tz_j|<1$, we have $\frac{1}{1-tz_j}=\sum_{k=0}^{\infty}t^kz_j^{k}$. It is sufficient to verify equivalence on the monomials. Thus, let $f(z)=z^{\alpha}$. We have $f(tz)=t^{|\alpha|}z^{\alpha}$. Thus  $f(tz)\frac{1}{1-tz_j}=t^{|\alpha|}z^{\alpha}\sum_{k=0}^{\infty}t^kz_j^{k}=\sum_{k=0}^{\infty}t^{|\alpha|+k}z^{\alpha+ke_j}$. We then have

\begin{align}
\int_0^1 f(tz) \frac{1}{1-tz_j} dt&=\int_0^1\left(\sum_{k=0}^{\infty}t^{|\alpha|+k}z^{\alpha+ke_j}\right) dt\\
\sum_{k=0}^{\infty}\left(\int_0^1 t^{|\alpha|+k} dt\right)z^{\alpha+ke_j}&= \sum_{k=0}^{\infty}\left(\frac{1}{|\alpha|+k+1}\right)z^{\alpha+ke_j}
\end{align}

as desired.

\end{proof}

\subsection{Decomposition into Shifts and Radial Derivatives}

Next, we decompose $C_j$ into previously studied operators. Recall for $f \in \mathcal{O}(\mathbb{B}_d)$ that the radial derivative is defined by

\begin{equation}
Rf=\sum_{j=1}^d z_j \frac{\partial f}{\partial z_j}
\end{equation}

from which it follows that

\begin{equation}
Rz^{\alpha}=|\alpha| z^{\alpha}.
\end{equation}

We can then formally define $(R+I)^{-1}z^{\alpha}=\frac{1}{|\alpha|+1}z^{\alpha}$. Likewise, recall the coordinate shift operator is given by $M_{z_j}z^{\alpha}=z^{\alpha+e_j}$. Iterating $k$ times gives $M_{z_j}^{k}z^{\alpha}=z^{\alpha+ke_j}$. Putting this together gives

\begin{equation}
C_j z^{\alpha}=\sum_{k=0}^{\infty} (R+I)^{-1} M_{z_j}^{k} z^{\alpha}=(R+I)^{-1} \sum_{k=0}^{\infty}  M_{z_j}^{k} z^{\alpha}.
\end{equation}

From this decomposition, one observes that the Ces\`aro tuple does not commute. For example, we have 

\begin{align}
C_iC_j 1&=C_i (R+I)^{-1} \sum_{k=0}^{\infty}M_{z_j}^k 1=C_i(R+I)^{-1}(1+z_j+z_j^2+\cdots)\\
&=C_i\left(1+\frac{1}{2}z_j+\frac{1}{3}z_j^2+\cdots \right)\\
&=(R+1)^{-1}\sum_{k=0}^{\infty}M_{z_i}^k\left(1+\frac{1}{2}z_j+\frac{1}{3}z_j^2+\cdots \right)\\
&=(R+1)^{-1}\sum_{k=0}^{\infty}\left(z_i^k+\frac{1}{2}z_i^kz_j+\frac{1}{3}z_i^kz_j^2+\cdots \right)\\
&=\sum_{k=0}^{\infty} \left(\frac{1}{k+1}z_i^k+\frac{1}{2}\frac{1}{k+2}z_i^kz_j+\frac{1}{3}\frac{1}{k+1}z_i^kz_j^2+\cdots \right).
\end{align}

Swapping $i$ and $j$ gives us the expression for $C_jC_i 1$. It is clear that for $i \neq j$, these expressions are not equal and thus they do not commute.

\subsection{The Inverse of $C_j$}

Let $g=C_jf$ with $f(z)=\sum_{\alpha}a_{\alpha}z^{\alpha}$ and $g(z)=\sum_{\beta}b_{\beta}z^{\beta}$. We observed in our proof of boundedness that we may write

\begin{equation}
b_{\beta}=\frac{1}{|\beta|+1}\sum_{k=0}^{\beta_j}a_{\beta-ke_j} \Rightarrow (|\beta|+1)b_{\beta}= \sum_{k=0}^{\beta_j}a_{\beta-ke_j}
\end{equation}

Thus, to find $C_j^{-1}$, we seek to solve $a_{\alpha}$ in terms of $b_{\beta}$. To isolate $a_{\alpha}$, note that

\begin{equation}
a_{\beta}=(|\beta|+1)b_{\beta}-|\beta|b_{\beta-e_j}.
\end{equation}

Thus we have

\begin{equation}
C_j^{-1}g(z)=\sum_{\beta} \left((|\beta|+1)b_{\beta}-|\beta|b_{\beta-e_j} \right)z^{\beta}.
\end{equation}

After some reindexing, a routine check verifies that we may write this in operator form as 

\begin{equation}
C_j^{-1}=(R+I)-RM_{z_j}=R(1-M_{z_j})+I.
\end{equation}

Additionally, using the commutation relation $RM_{z_j}=M_{z_j}R+M_{z_j}$, we may write this as 

\begin{equation}
C_j^{-1}=(I-M_{z_j})(R+I).
\end{equation}

We observe thus, due to that radial derivative, that $C_j^{-1}$ is unbounded.

\subsection{Spectral Properties of the Ces\`aro Tuple and Noncompactness}

There is an established generalization for the spectrum of a tuple of operators known as the Taylor spectrum \cite{Taylor}. However, it requires the operators in the tuple to commute. Since the operators in our Ces\`aro tuple do not commute, we cannot study the Taylor spectrum. However, one may still study the classical spectrum for a given $C_j$.

Let $M_{\alpha}= \overline{\text{span}}\{z^{\alpha+ke_j}, k \geq 0, \alpha_j=0\}$. Observe that $C_j$ maps each $M_{\alpha}$ into itself. Additionally, we note that the Drury-Arveson space may be written as a direct sum given by 

\begin{equation}
H^2_d(\mathbb{B}_d)=\bigoplus_{\alpha \mid \alpha_j=0}M_{\alpha}.
\end{equation}

We can observe this explicitly via 

\begin{equation}
C_jz^{\alpha}=C_j \left(z_j^{\alpha_j} \prod_{i \neq j} z_i^{\alpha_i} \right)=\left(\sum_{k=0}^{\infty}\frac{1}{\alpha_j+|\alpha'|+k+1} z^{\alpha_j+k}\right) \left( \prod_{i \neq j} z_i^{\alpha_i} \right)
\end{equation}

where $|\alpha'|=\sum_{i \neq j} \alpha_i$. Next, recall that the spectrum of a direct sum is the union of the the spectrum of each subspace. Hence we compute $\sigma (M_{\alpha})$ and may consider our eigenfunctions $f$ as functions of $f(z_j)$ by fixing $\alpha'$ (i.e. $\alpha$ minus $\alpha_j$). We will formally solve the eigenalue equation $C_jf=\lambda f$ and then determine which $\lambda$'s correspond to eigenfunctions in $H^2_d$.

On $M_{\alpha}$, where $\alpha'$ is fixed, our integral representation may be determined as follows. Let $u=z_jt$ so that $dt=\frac{du}{z_j}$ and $t^M=\left(\frac{u}{z_j}\right)^M$ where $M=|\alpha'|$. Then we have

\begin{align}
(C_jf)(z)&=\int_0^1\frac{f(tz)}{1-tz_j}dt=\int_0^1\frac{f(tz_j)}{1-tz_j}t^Mdt\\
&=\frac{1}{z_j^{M+1}}\int_0^{z_j}\frac{f(u)u^M}{1-u}du
\end{align}

We then consider $C_jf= \lambda f$. This gives 

\begin{equation}
\lambda f(z_j)=\frac{1}{z_j^{M+1}}\int_0^{z_j}\frac{f(u)u^M}{1-u}du.
\end{equation}

After multiplying both sides of our equation by $z_j^{M+1}$ and differentiating with respect to $z_j$, we obtain

\begin{equation}
\lambda(M+1)z_j^Mf(z_j)+\lambda z_j^{M+1}f'(z_j)=\frac{f(z_j)z_j^M}{1-z_j}.
\end{equation}

This simplifies to 

\begin{equation}
\frac{f'(z_j)}{f(z_j)}=\left(\frac{1}{\lambda}-M-1 \right)\frac{1}{z_j}+\frac{1}{\lambda}\frac{1}{1-z_j}.
\end{equation}

which, after integrating and solving for $f(z_j)$ gives

\begin{equation}
f(z_j)=z_j^{\frac{1}{\lambda}-M-1}(1-z_j)^{-\frac{1}{\lambda}}
\end{equation}

Next, let $w=\frac{1}{\lambda}$ so that we have $(1-z_j)^{-w}=\sum_{k=0}^{\infty}\frac{\Gamma(k+w)}{\Gamma(k+1)\Gamma(w)}z_j^k$. It is well known that this converges precisely when $\Re w < \frac{1}{2}$ and hence is when $f$ is in our space. Writing $\lambda=x+iy$, we have $\Re w=\frac{x}{x^2+y^2}$. Setting this equal to $\frac{1}{2}$ and doing a little algebra gives $(x-1)^2+y^2=1$ from which it follows that 

\begin{equation}
\sigma(C_j)=\{\lambda \in \mathbb{C} \mid |\lambda-1| \leq 1 \}.
\end{equation}

Furthermore, by basic properties of the spectrum of compact operators (e.g. the closed disk is an uncountable set), it follows that $C_j$ is not compact.

\subsection{The Adjoint}

We next compute the adjoint of $C_j$. Since the monomials form an orthogonal basis, it is sufficient for our purposes to compute the adjoint acting on $z^{\beta}$.

\begin{Theorem}
The adjoint of $C_j$ acting on the monomial $z^{\beta}$ is given by

\begin{equation}
C_j^*z^{\beta}=\frac{\beta !}{(|\beta|+1)!}\sum_{k=0}^{\beta_j}\left(\frac{(|\beta|-k)!}{(\beta-ke_j)!} \right) z^{\beta-ke_j}.
\end{equation}

\end{Theorem}

\begin{proof}
By definition we have 

\begin{align}
\langle C_j z^{\alpha}, z^{\beta} \rangle&=\langle z^{\alpha}, C_j^*z^{\beta} \rangle=\sum_{k=0}^{\infty} \frac{1}{|\alpha|+k+1}\langle z^{\alpha+ke_j}, z^{\beta} \rangle\\
&=\sum_{k=0}^{\infty}\frac{1}{|\alpha|+k+1}\frac{(\alpha+ke_j)!}{(|\alpha|+k)!}.
\end{align}

Hence $\beta=\alpha+ke_j$ which implies $\alpha=\beta-ke_j$ from which it follows that $\beta_j \geq k$. We thus write 

\begin{equation}
\langle C_j z^{\beta-ke_j}, z^{\beta} \rangle=\langle \frac{1}{|\beta-ke_j|+k+1}z^{\beta}, z^{\beta} \rangle=\frac{1}{|\beta|+1}||z^{\beta}||^2
\end{equation}

Next, since $\beta_j \geq k$, we write our adjoint as a linear combination of our monomials $\beta-ke_j$ ranging from $0$ to $\beta_j$. This gives

\begin{equation}
C_j^*z^{\beta}=\sum_{k=0}^{\beta_j} b_k z^{\beta-ke_j} \Rightarrow \langle z^{\beta-ke_j}, C_j^*z^{\beta} \rangle=b_k ||z^{\beta-ke_j}||^2.
\end{equation}

Equating expressions gives us $b_k=\frac{\beta !}{(|\beta|+1)!}\frac{(|\beta|-k)!}{|\beta|!}$ which, when substituting into our summation, gives our result.

\end{proof}

\section{The Roadmap to Future Discovery}

While we have resolved several immediate properties of our Ces\`aro tuple, much remains to be done. As a prompt for the interested reader, below are several avenues of investigation to pursue in regards to our tuple.

\begin{itemize}
\item We observed that $H^2_d(\mathbb{B_d})$ is one space on a natural scale of Hilbert function spaces. How do these (and future) results generalize to values of $\alpha \neq 1$?

\item Many questions involving  $H^2_d(\mathbb{B_d})$ are more easily addressed when considered in a noncommutative setting \cite{Jury}. What happens if we recast the Ces\`aro tuple problems in this setting? 

\item In one variable, the Ces\`aro operator has many deep relationships with other popular matrices such as the Hilbert and Hausdorff matrices \cite{Bennett1}, \cite{Bennett2}, can we construct analogous relations in several variables?

\item It is known that the Ces\`aro operator in one variable is hyponormal on $H^2$, i.e. $A^*A-AA^* \geq 0$ \cite{Halmos} and subnormal on $H^2$ \cite{Kriete}. Do these results generalize to the operators in our Ces\`aro tuple?

\end{itemize}

These are just a few of the many directions one may take the results of this article. It is up to the interested reader which road to take next...

\section*{Acknowledgements}
 The author would like to thank Alan Sola for words of encouragement, Michael Hartz for helpful comments on the Drury-Arveson space, and Bill Ross for writing an excellent exposition on the Ces\`aro operator that motivated this paper.

\end{document}